\documentclass[12pt,twoside,leqno,a4paper]
{amsart}
\usepackage{amsmath}
\usepackage{amssymb} 
\usepackage[english]{babel}
\usepackage{color}
\usepackage{bbm}

\theoremstyle{plain}
\newtheorem{thm}{Theorem}[section]
\newtheorem{lem}[thm]{Lemma}

\newtheorem{prop}[thm]{Proposition}
\newtheorem{cor}[thm]{Corollary}
\newtheorem{rem}[thm]{Remark}

\newtheorem*{Chernoff:inequality}
{Chernoff's inequality}

\newcommand{\C}{{\mathbb C}}

\newcommand{\R}{{\mathbb R}}

\newcommand{\N}{{\mathbb N}}

\newcommand{\Sn}{\mathbb{S}^n}
\newcommand{\Bn}{\mathbb{B}^n}

\newcommand\pla{^{p}_{\alpha}}

\renewcommand{\Re}{\operatorname{Re}}

\def\z{\zeta}

\title[Boundedness of the Bergman projection on 
 Fock spaces]{Boundedness of the Bergman 
 projection on generalized Fock-Sobolev spaces on 
 $\C^n$}

\author{Carme Cascante}
\address{C. Cascante: Departament de Matem\`atiques i
    Inform\`atica, Universitat  de Barcelona,
     Gran Via 585, 08071 Barcelona, Spain}
\email{cascante@ub.edu}
\author{Joan F\`abrega}
\address{J. F\`abrega: Departament de Matem\`atiques i
    Inform\`atica, Universitat  de Barcelona,
     Gran Via 585, 08071 Barcelona, Spain}
\email{joan$_{-}$fabrega@ub.edu}
\author{Daniel Pascuas}
\address{D. Pascuas: Departament de Matem\`atiques i
    Inform\`atica, Universitat  de Barcelona,
     Gran Via 585, 08071 Barcelona, Spain}
\email{daniel$_{-}$pascuas@ub.edu}


\date{\today}

\thanks{The research 
	was supported in part by
        Ministerio de Econom\'{\i}a y Competitividad, Spain,   projects MTM2014-51834-P and MTM2015-69323-REDT, and Generalitat de Catalunya, project
2014SGR289. The first author was also supported in part by Ministerio de Econom\'{\i}a y Competitividad, Spain,   project MDM-2014-0445}

\begin{document}

\begin{abstract}

In this paper we solve a problem posed by 
H. Bommier-Hato, M. Engli\v{s} and E.H. Youssfi 
in  \cite{bommier-englis-youssfi}
on the boundedness of the Bergman-type projections
in generalized Fock spaces. 
It will be a consequence of two facts:
a full  description of the embeddings  between 
 generalized Fock-Sobolev  spaces and a complete 
 characterization of the boundedness of the above Bergman type projections
 between  weighted $L^p$-spaces related to generalized  Fock-Sobolev spaces. 
\end{abstract}
\maketitle	

\section{Introduction}
Let $dV=dV_n$ be the Lebesgue measure on $\C^n$ 
 normalized so that the measure of the unit ball $\Bn$ 
 is 1. If $n=1$ we write $dA=dV_1$. 
Let $d\sigma$ be the Lebesgue measure on the unit 
 sphere $\Sn$ normalized so that  $\sigma(\Sn)=1$. 
We denote by $H=H(\C^n)$ the space of entire 
 functions on $\C^n$.

Let $\ell>0$. For $1\le p<\infty$,  $\alpha>0$ and 
 $\rho\in\R$, 	the space 
$L^{p,\ell}_{\alpha, \rho}=L^{p}_{\alpha, \rho}$ consists 
of all measurable functions $f$ on $\C^n$ 
 such that
	\[
	\|f\|^p_{L^{p}_{\alpha, \rho}}:=
	\int_{\C^n} \bigl|f(z)(1+|z|)^{\rho }
e^{-\frac \alpha 2 |z|^{2\ell}}\bigr|^pdV(z)<\infty,
	\]
	that is, $L^{p}_{\alpha, \rho}=L^p(\C^n;
(1+|z|)^{\rho p}e^{-\frac{\alpha p}2|z|^{2\ell}}dV(z))$.
	
Moreover,
 $L^{\infty,\ell}_{\alpha,\rho}=L^{\infty}_{\alpha,\rho}$
 consists of all measurable functions $f$ on $\C^n$
 such that
	\[
	\|f\|_{L^{\infty}_{\alpha,\rho}}=
	\operatorname*{ess\,sup}_{z\in\C^n}|f(z)|
(1+|z|)^{\rho}e^{-\frac \alpha 2|z|^{2\ell}}<\infty.
	\]

	We define the generalized Fock-Sobolev spaces as 
 $F^{p}_{\alpha, \rho}:=H\cap L^{p}_{\alpha, \rho}$. 

When $\rho=0$, we obtain the generalized Fock spaces
 $F^{p}_{\alpha}=F^{p}_{\alpha, 0}$. 
According to this notation we write
 $L^{p}_{\alpha}=L^{p}_{\alpha, 0}$.

The space $L^{2}_\alpha$ is a Hilbert space with
 the inner product 
\[
\langle f,g\rangle_\alpha:=\int_{\C^n}
 f(z)\overline{g(z)}e^{-\alpha|z|^{2\ell}}dV(z).
\]
and $F^{2}_\alpha$ is a closed linear subspace of
 $L^{2}_\alpha$. Denote by $P_\alpha$ the  orthogonal projection from
 $L^{2}_\alpha$ to $F^{2}_\alpha$, which is usually
 called the Bergman projection.

In \cite[Theorem 9.1]{janson-peetre-rochberg} 
the authors showed that $P_\alpha$ is bounded from
$L^{p}_{\beta}$ to $F^{p}_{\gamma}$ 
if and only if $\beta<2\alpha$ and $\beta=\gamma$. 
In \cite{bommier-englis-youssfi} the authors studied 
the boundedness 
of  $P_\alpha$ 
between  the spaces 
$\mathcal{L}^p_b:=L^p(\C^n; e^{-b|z|^{2\ell}}dV(z))$
 and 
$\mathcal{L}^q_d:=L^q(\C^n; e^{-d|z|^{2\ell}}dV(z))$.
Observe that $\mathcal{L}^p_a=L^{p}_{2a/p}$. 
Since $\mathcal{L}^2_a=L^{2}_{a}$ the orthogonal 
 projection  $ \mathcal{P}_a$ from 
 $\mathcal{L}^2_a$ onto 
 $\mathcal{F}^2_a:=H\cap\mathcal{L}^2_a$ 
 coincides with $P_a$.
One advantage of considering the spaces  
 $L^{p}_{\alpha}$ is that  permits us to include the 
 case $p=\infty$. 
Their results are given in terms of a parameter $c$
 defined by $c:=\frac{4d}{a^2 q}(a-\frac bp)$. 
Rewriting the parameters as $a=\alpha$, 
$b=\beta p/2$ and $d=\gamma q/2$, 
we have that, in our notations, 
$c=\gamma\frac{2\alpha-\beta}{\alpha^2}$. 

The main results in \cite{bommier-englis-youssfi} are:

\begin{enumerate}
\item If  $P_\alpha$ is bounded then $c\ge 1$.

\item If 
$c>1$ then $P_\alpha$ is bounded.

\item  If 
$c=1$ and $\ell\le 1$ then $P_\alpha$ is bounded if 
and only if $q\ge p$.
\end{enumerate}

For  $c=1$ and $\ell>1$ 
 the authors only obtain  partial results.
In particular they prove that if $c=1$ and 
 $\frac{2n}{2n-1}<\ell<2$ 
then $P_\alpha$ is bounded if and only if $q=p$.

The initial motivation of this work was to close the 
 remaining open cases which will be achieved by 
 proving:

\begin{enumerate}
\item[(iv)] If $c=1$ and $\ell>1$ then $P_\alpha$ is bounded if and only if $q=p$.
\end{enumerate}

This result shows that, of the four possible mutually 
 exclusive assertions in 
\cite[Proposition 17]{bommier-englis-youssfi},
(a) is the valid option.

Note that  if $c\ge 1$, then
$a-\frac bp>0$, which 
in our notation is equivalent to 
 $\beta<2\alpha$. The later condition is necessary in 
 order that the "pointwise evaluation" of the 
 Bergman projection is bounded on $L^p_{\beta}$ 
 (see Lemma \ref{lem:well-defined} below).

Our main result is the following theorem 
for generalized Fock-Sobolev spaces.

  \begin{thm}\label{thm:Bprojection}
  	Let $\ell\ge 1$, $\alpha, \beta, \gamma>0$ and 
 $\rho,\eta\in\R$. For $1\le p,q\le \infty$, 
  	 $P_\alpha$ maps boundedly
  	$L^{p}_{\beta,\rho}$ to 
  	$L^{q}_{\gamma,\eta}$ if and
  	only if  one of the 
  	following conditions holds:
  	\begin{enumerate}
  		\item $0<\alpha^2/(2\alpha-\beta)<\gamma$.
  		\item   $\alpha^2/(2\alpha-\beta)=\gamma$,
  		$p\le q$ and 
  	$\rho-\eta\ge 2n(\ell-1)\left(\frac 1p-\frac 1q  \right)$.
  		\item $\alpha^2/(2\alpha-\beta)=\gamma$,
  		$q< p$ and  $\rho-\eta> 2n 
  		\left(\frac 1q-\frac 1p \right)$.
  	\end{enumerate}
  \end{thm}

In particular for $\rho=\eta$ we obtain the following 
 generalization of (iv).

\begin{cor}\label{cor:Bprojection}
	Let $\ell> 1$, $\alpha, \beta, \gamma>0$ and 
 $\rho\in\R$.	For $1\le p,q\le \infty$, 
	$P_\alpha$ maps boundedly
	$L^{p}_{\beta,\rho}$ to 
	$L^{q}_{\gamma,\rho}$ if and
	only if  either
	$0<\alpha^2/(2\alpha-\beta)<\gamma$ or 
 $\alpha^2/(2\alpha-\beta)=\gamma$ and $p=q$.
	\end{cor}

Our approach to obtain 
Theorem \ref{thm:Bprojection}  differs from the one in  
\cite{bommier-englis-youssfi}. Instead of proving 
directly the characterizations, we deduce the results 
as a consequence of two ingredients: 
the first is the identity (see Proposition 
 \ref{prop:Ponto}  below)
 \begin{equation}\label{eqn:PLontoF}
 P_\alpha(L^{p}_{\beta,\rho})
=F^{p}_{\frac{\alpha^2}{2\alpha-\beta},\rho}\quad
(1\le p\le\infty,  \ell\ge 1,  \beta<2\alpha, \rho>0)
\end{equation}
and the second one is the following embedding result:
\begin{thm}\label{thm:embeddings}
	Let $\ell\ge 1$, $\beta,\gamma>0$ and 
 $\rho,\eta\in\R$. For $1\le p,q\le\infty$,
the embedding $F^{p}_{\beta,\rho}\hookrightarrow 
 F^{q}_{\gamma,\eta}$ holds
	if and only if  one of the following  three conditions is 
 satisfied:
	\begin{enumerate}
		\item\label{item:embeddings1} $\beta<\gamma$.
		\item\label{item:embeddings2} $\beta=\gamma$, 
 $q\ge p$ and \,\,	
$2n(\ell-1)\left(\frac 1p-\frac 1q\right)\le \rho-\eta$.
		\item\label{item:embeddings3} $\beta=\gamma$,
 $q< p$ and \,\,
$2n\left(\frac 1q-\frac 1p\right)< \rho-\eta$.
	\end{enumerate}
\end{thm}

Note that as an immediate consequence of 
Theorem \ref{thm:embeddings} we obtain:

\begin{cor}$ $

\begin{enumerate}
\item 
If $\ell\ge 1$ and
the embedding $F^{p}_{\beta,\rho}\hookrightarrow 
 F^{q}_{\beta,\eta}$ holds, 
then $\rho\ge \eta$.
\item For $\ell=1$,   the embedding
$F^{p}_{\beta,\rho}\hookrightarrow 
 F^{q}_{\beta,\rho}$  holds
if and only if $p\le q$.
\item\label{item3} 
For $\ell>1$,   the embedding
$F^{p}_{\beta,\rho}\hookrightarrow F^{q}_{\beta,\rho}$ 
if and only if $p= q$.
	\end{enumerate}
\end{cor}

The proof of Theorem \ref{thm:embeddings} requires
 of some results which can be of interest by 
 themselves. 
For instance, assertions  \eqref{item:embeddings1} 
 and \eqref{item:embeddings2} 
follow from precise pointwise and 
 $L^{p}_{\beta,\rho}$-norm estimates of 
the Bergman kernel. 
As a consequence, we derive pointwise
 estimates of the functions in $F^{p}_{\beta,\rho}$ and 
some properties on the boundedness of the Bergman 
 projection.
The most difficult part is the proof of assertion 
 \eqref{item:embeddings3}. In this case, 
for $1\le q<p<\infty$, we use a 
technique due to D. Luecking (see \cite{luecking}), 
 based on Kinchine's  inequality, 
which permits the construction of adequate 
test functions. 
Then  the case  $1\le q<p=\infty$ 
follows by extrapolation.

The paper is organized as follows: 
In Section \ref{sec:Bergman} we obtain pointwise
 and $L^p_{\alpha,\rho}$-norm estimates of the
 Bergman kernel, from which 
 the boundedness
 of the Bergman projection $P_\alpha$ on
 $L^p_{\alpha,\rho}$ is deduced. 
In Sections \ref{sec:embeddings} and \ref{sec:boundedness} we prove 
Theorems \ref{thm:embeddings} and 
 \ref{thm:Bprojection} respectively.

{\bf Notations:}
In  the next sections we only consider spaces 
 $F^{p,\ell}_{\alpha,\rho}=F^{p}_{\alpha,\rho}$, with 
 $\ell\ge 1$, $\alpha>0$ and $\rho\in\R$. 
So we omit the conditions on $\ell,\alpha$ and $\rho$ 
 in the statement of the results.
We denote by $p'$ the conjugate exponent of $p\in [1,\infty]$.

Let $\N$ be the set of non-negative entire
 numbers.
 For a multi-index $\nu=(\nu_1,\cdots,\nu_n)\in \N^n$
  and $z=(z_1,\cdots,z_n)\in \C^n$,    we write,  as usual,
  $z^\nu=z_1^{\nu_1}\cdots z_n^{\nu_n}$,
$\nu!=\nu_1!\cdots\nu_n!$ and 
 $|\nu|=\nu_1+\cdots+\nu_n$.

For $z,w\in\C^n$, 
$z\overline{w}=\sum_{j=1}^n z_j\overline{w}_j$. 
If $z\in\C^n$ and $r>0$ then $B(z,r)$ is the open ball in
 $\C^n$ with center~{$z$} and radius $r$. 
When $n=1$, $B(z,r)$ is denoted, as usual, by $D(z,r)$.

If $E\subset\C^n$ then $\mathcal{X}_{E}$ is the 
characteristic function of $E$. 
 
If $X, Y$ are normed spaces, the notation
 $X\hookrightarrow Y$ means that the mapping
$f\in X\mapsto f\in Y$ is bounded.

 For $\lambda\in\C\setminus\{0\}$, we denote by 
 $\arg\lambda$   the principal branch of the 
 argument of $\lambda$,  that is, 
 $-\pi<\arg\lambda\le\pi$. 
Moreover, 
 $\lambda^\beta=|\lambda|^\beta
e^{i\beta\arg\lambda}$, for $\beta\in\R$.

The letter $C$ will denote a positive constant, 
which may vary from place to place. 
The notation $\Phi\lesssim \Psi$ means that there
 exists a constant $C>0$, which does not depend on
 the involved variables,
 such that $\Phi\le C\, \Psi$. We write $\Phi\simeq \Psi$
 when $\Phi\lesssim \Psi$ and $\Psi\lesssim \Phi$.

\section{The Bergman projection on
 $L^{p}_{\alpha,\rho}$}\label{sec:Bergman}

\subsection{On the two parametric Mittag-Leffler
 functions $E_{a,b}$}\quad\par
 The two parametric  Mittag-Leffler  functions are the
 entire functions
\[
E_{a,b}(\lambda)
:=\sum_{k=0}^\infty \frac{\lambda^k}{\Gamma(a k+b)}
\qquad(\lambda\in\C,\,\, a,b>0).
\]
A good general reference for the Mittag-Leffler
 functions is the
 book~{\cite{gorenflo-kilbas-mainardi-rogosin}}.

In this section we recall the asymptotic expansions of 
the two parametric Mittag-Leffler functions and their
 derivatives.
Those expansions will be useful to obtain pointwise
 and norm estimates of the Bergman kernel.

\begin{thm}[{\cite[Theorem 1.2.1]{popov-sedletskii}}]
	Let $a\in (0,1)$ and $b>0$. 
Then, for $|\lambda|\to\infty$, we have
\begin{equation}\label{eqn:Eab}
E_{a,b}(\lambda)=
\begin{cases}	
\frac 1a \lambda^{(1-b)/a}e^{\lambda^{1/a}}
+O(\lambda^{-1}),
	 & \quad\text{if}\quad|\arg\lambda|\le a\pi,\\
O(\lambda^{-1}),
	&  \quad\text{if}\quad|\arg\lambda|\ge a\frac{2\pi}3.
	\end{cases}
	\end{equation}
\end{thm}

By Cauchy formula (see~\cite[Theorem 1.4.2]{olver}), 
 the 
asymptotic expansions of the $m$-th derivatives of $E_{a,b}$
 (on ``smaller''  sectors that the ones involved in~{\eqref{eqn:Eab}}) 
 can be obtained by differentiating $m$ times the terms in~{\eqref{eqn:Eab},} that is,
\begin{equation}\label{eqn:derEab}
E_{a,b}^{(m)}(\lambda)=
\begin{cases}	
\frac 1a \frac{d^m}{d \lambda^m}\left(\lambda^{(1-b)/a} e^{\lambda^{1/a}}\right)
+O(\lambda^{-1-m}),
& \text{if}\quad|\arg\lambda|\le a\frac{3\pi}4,\\
O(\lambda^{-1-m}),
& \text{if}\quad|\arg\lambda|\ge a\frac{3\pi}4.
\end{cases}
\end{equation}

\subsection{The Bergman kernel}\quad\par

The next result, which is obtained in \cite{bommier-englis-youssfi}, gives a description of the Bergman kernel. 
The main tool to compute the norm of the monomials in $F^{2}_{\alpha}$ is the identity
\[
\Gamma(x)=\int_0^\infty t^{x-1}e^{-t }dt=
2\ell \gamma^x \int_0^\infty s^{2\ell x-1}
e^{-\gamma s^{2\ell}}ds\quad (x>0,\,\gamma>0).
\]

\begin{lem}
The system 
$\bigl\{\frac{w^\nu}{\|w^\nu\|_{F^{2}_\alpha}}\bigr\}_{\nu\in
\N^n}$ is an orthonormal basis for  $F^{2}_\alpha$, so the Bergman projection from $L^2_\alpha$ onto $F^2_\alpha$ is
\[
P_\alpha f(z)=\langle f,K_{\alpha,z}\rangle_\alpha  =\int_{\C^n}f(w)K_{\alpha}(\z,w)e^{-\alpha |w|^{2\ell}}dV(w),
\]
where 
	\[
	K_\alpha(z,w)=\overline{K_{\alpha,z}(w)}
	=\sum_{\nu\in \N^n}\frac{z^\nu\overline{w}^\nu}{\|w^\nu\|^2_{F^{2}_{\alpha}}}
	\]
	is the Bergman kernel.
	Namely, since 
	$\|w^\nu\|^2_{F^{2}_{\alpha}}=
	\frac 1\ell\frac{n!\, \nu!\, \Gamma\left(\frac{|\nu|+n}\ell\right)}{(n-1+|\nu|)!}
	$,
$
K_{\alpha}(z,w)=H_{\alpha}(z\overline{w})
$, where
\[
H_{\alpha}(\lambda)
:=\frac{\ell\alpha^{n/\ell}}{n!}
\sum_{k=0}^\infty \frac{(n-1+k)!}{k!}\frac{\alpha^{k/\ell}\lambda^k}
{\Gamma\left(\frac{k+n}\ell\right)}
=\frac{\ell\alpha^{n/\ell}}{n!}
E_{1/\ell,1/\ell}^{(n-1)}(\alpha^{1/\ell}\lambda).
\]	
In particular, for any $\delta>0$ we have 
\begin{equation}\label{eqn:Kdelta}
K_{\alpha}(z,\delta w)
=\delta^{-n}K_{\alpha\delta^\ell }(z,w).
\end{equation}
\end{lem}

\begin{rem}\label{rem:pointwB}
In order to obtain norm estimates of the Bergman kernel it is useful to make the following change of variables.
Given $z\in\C^n$, there is a unitary mapping  $U:\C^n\to\C^n$ such that $U(z)=(|z|,0,\,\dots,0)$.
Then  
$
K_{\alpha}(w,z)=H_{\alpha}(|z|u_1)
$,
where $U(w)=(u_1,\cdots,u_n)$,
so we may assume $z=(|z|,0,\cdots,0)$.
\end{rem}

The remaining part of this section is devoted to derive
 pointwise and norm estimates of the Bergman kernel, 
which will be the key tools to obtain our main results.

The following corollaries are consequences of 
 \eqref{eqn:derEab}.

\begin{cor}\label{cor:diffE}
Let $n$ be a positive integer. For $|\lambda|\to\infty$, we have that
\[
E^{(n-1)}_{1/\ell,1/\ell}(\lambda)=
\begin{cases}	\ell^n \lambda^{n(\ell-1)}e^{\lambda^\ell}
(1+O(\lambda^{-\ell}))+O(\lambda^{-n}),
& \text{if}\quad|\arg \lambda|\le \frac{3\pi}{4\ell},\\
O(\lambda^{-n}),
& \text{if}\quad|\arg \lambda|\ge \frac{3\pi}{4\ell}.
\end{cases}
\]
\end{cor}

\begin{proof}

For $\ell=1$, $E_{1/\ell,1/\ell}(\lambda)=e^\lambda$ so $E^{(n-1)}_{1/\ell,1/\ell}(\lambda)=e^\lambda$, and the above asymptotic  identity is obvious in this case.

Next assume $\ell>1$.
	By induction on $n$ it is easy to check that
	\[
	\ell\,\frac{d^{n-1}}{d\lambda^{n-1}}\lambda^{\ell-1}
	e^{\lambda^\ell}
	=\ell^n \lambda^{n(\ell-1)}e^{\lambda^\ell}
	(1+O(\lambda^{-\ell}))
	\quad(|\lambda|\to\infty,\,|\arg\lambda|<\pi/\ell).
	\]

	By combining this identity with \eqref{eqn:derEab} we obtain the result.
	\end{proof}

\begin{cor}\label{cor:pointwB}
For any $\delta>0$ and $N>2$, let
$S^{\delta}_N:=D(0,\delta)\cup S_N$,
where
\[
S_N:=\{0\}\cup
\{\,\lambda\in\C\setminus\{0\} :
 |\arg\lambda|\le\tfrac{\pi}{N\ell}\,\}.
\]
Then there exist $\delta>0$ and $N>2$ such that 
	\begin{align}
	&|H_{\alpha}(\lambda)|
	\simeq(1+|\lambda|)^{n(\ell-1)}\,
\bigl|e^{\alpha\lambda^\ell}\bigr|
	\qquad (\lambda\in S^{\delta}_N),
     \label{eqn:simeq:estimate:H}\\	
	& |H_{\alpha}(\lambda)|\lesssim
	 (1+|\lambda|)^{n(\ell-1)}\,e^{\alpha\cos(\frac{\pi}N)|\lambda|^\ell}
	\qquad (\lambda\in \C\setminus S^{\delta}_N).
     \label{eqn:lesssim:estimate:H}	
\end{align}
In particular, 
\begin{equation}\label{eqn:rough:estimate:H}
\mathcal{X}_{S_N}(\lambda)\lesssim|H_{\alpha}(\lambda)|\lesssim
(1+|\lambda|)^{n(\ell-1)}\,e^{\alpha|\lambda|^\ell}
\qquad(\lambda\in\C).
\end{equation}
	\end{cor}

\begin{proof}
Corollary~{\ref{cor:diffE}} shows that there is a large $R>0$ so that
\begin{align}
&|H_{\alpha}(\lambda)|\simeq
(1+|\lambda|)^{n(\ell-1)}\bigl|e^{\alpha\lambda^\ell}\bigr|
\qquad (|\lambda|\ge R,\,|\arg\lambda|\le\tfrac{\pi}{3\ell}),
 \label{eqn:simeq:estimate:H:0}\\
&|H_{\alpha}(\lambda)|\lesssim
(1+|\lambda|)^{n(\ell-1)} e^{\frac{\alpha}2|\lambda|^\ell}
\qquad (|\lambda|\ge R,\,|\arg\lambda|\ge\tfrac{\pi}{3\ell}).
 \label{eqn:lesssim:estimate:H:0}	
\end{align}
Since $H_{\alpha}$ is a continuous positive function on 
the interval $[0,\infty)$, we have that there exist a small 
$\delta>0$ and a large $N>2$ such that
\begin{equation} \label{eqn:simeq:estimate:H:1}
|H_{\alpha}(\lambda)|\simeq 1\simeq (1+|\lambda|)^{n(\ell-1)}\bigl|e^{\alpha\lambda^\ell}\bigr|
\qquad(\lambda\in S^{\delta}_N,\,|\lambda|\le R).
\end{equation}	
Therefore~{\eqref{eqn:simeq:estimate:H}} directly follows 
from~{\eqref{eqn:simeq:estimate:H:0}}
 and~{\eqref{eqn:simeq:estimate:H:1}}.
Moreover, \eqref{eqn:lesssim:estimate:H} is deduced from~{\eqref{eqn:simeq:estimate:H:0}}, \eqref{eqn:lesssim:estimate:H:0} and the fact that $H_{\alpha}$ is bounded on $D(0,R)$.  	
\end{proof}

As an immediate consequence of the above 
results we obtain the following pointwise estimate for
 the Bergman kernel.

\begin{prop}\label{prop:pointK}
There exist $\delta>0$ and $N>2$ such that
	\begin{align}
	&|K_{\alpha}(w,z)|
	\simeq(1+|z\overline w|)^{n(\ell-1)}\,e^{\alpha\Re((z\overline w)^\ell)}
	\quad (z\overline w\in S^{\delta}_N),
     \label{eqn:simeq:estimate:K}\\	
	& |K_{\alpha}(w,z)|\lesssim
	 (1+|z\overline w|)^{n(\ell-1)}\, e^{\alpha\cos(\frac{\pi}N)|z\overline w|^\ell}
	\quad (z\overline w\in \C\setminus S^{\delta}_N).
     \label{eqn:lesssim:estimate:K}	
\end{align}
\end{prop}

Now we state norm estimates for the Bergman kernel.

\begin{prop}\label{prop:pnormBergman}
	Let $1\le p\le\infty$. Then
	\[
	\|K_{\alpha}(\cdot,z)\|_{F^{p}_{\alpha,\rho}}\simeq (1+|z|)^{\rho +2n(\ell-1)/p'}e^{\frac{\alpha}2|z|^{2\ell}}
	\quad(z\in\C^n).
	\]
\end{prop}

This estimate for $1\le p<\infty$ and $\rho=0$ is stated without a 
 detailed proof in 
\cite[Section 8.1]{bommier-englis-youssfi}. 
Since this norm estimate of the Bergman kernel is essential in order to obtain 
our main theorems and it is deduced from several non-trivial technical results, 
we include its proof.
The main tool  is the pointwise estimate of 
$H_{\alpha}$ given in Corollary~{\ref{cor:pointwB}}, 
but we also need the following three 
technical lemmas. 

\begin{lem} \label{lem:est:sup}
Let $\alpha>0$ and let $\beta\in\R$. Then 
\[
\sup_{x\ge0}\,(1+x)^{\beta}e^{-\alpha(x-a)^2}\simeq (1+a)^{\beta}\qquad(a\ge0).
\]
\end{lem}

\begin{proof} 
Since
$(1+x)^{\beta}e^{-\alpha(x-a)^2}=
((1+x)^{\beta/\alpha}e^{-(x-a)^2})^{\alpha}$, for any $a,x\ge0$,
 we may assume that $\alpha=1$. 
Then it is clear that
$\sup_{x\ge0}\,(1+x)^{\beta}e^{-(x-a)^2}\ge (1+a)^{\beta}$,
 for every $a\ge0$, and so we only have to prove that
\[
(1+x)^{\beta}e^{-(x-a)^2}\lesssim (1+a)^{\beta}\qquad(a,x\ge0).
\]
Let $x\ge0$. If $a-(1+a)/2\le x$ then $-1/2\le(x-a)/(1+a)$ and so, for any $\beta\in\R$,
\[
(1+x)^{\beta}e^{-(x-a)^2}
\le (1+a)^{\beta}\,\Bigl(1+\frac{x-a}{1+a}\Bigr)^{\beta}
          e^{-\left(\frac{x-a}{1+a}\right)^2} 
\le  (1+a)^{\beta}
\sup_{t\ge-1/2}(1+t)^{\beta}e^{-t^2}.           
\]

Next assume $x\le a-(1+a)/2$. If $\beta<0$ then 
\[
(1+x)^{\beta}e^{-(x-a)^2}\le e^{-(x-a)^2}\le e^{-\frac14(1+a)^2}
\le (1+a)^{\beta}\, \sup_{t\ge1} t^{-\beta}e^{-t^2/4}.
\] 
Finally, if $\beta\ge0$  then 
$(1+x)^{\beta}e^{-(x-a)^2}\le(1+x)^ {\beta}\le(1+a)^{\beta}$.
\end{proof}	

\begin{lem} \label{lem:est:Cm}
	Let $a>0$ and let $b\in\R$. Then
	\[
	\int_{\C^{n-1}} (1+y+|w|)^b
	e^{-a(y^2+|w|^2)^{\ell}} dV_{n-1}(w)
	\simeq
	(1+y)^{b-2(n-1)(\ell-1)}
	e^{-a\,y^{2\ell}}\quad (y\ge 0).	
	\]
\end{lem}

\begin{proof} It is clear that the estimate of the statement holds
 for $0\le y\le 1$. Thus, by integration in polar coordinates, we 
 only have to prove that 
	  	\[
	  	I(y):=
	  \int_{0}^\infty (y+r)^b
	  e^{-a(y^2+r^2)^{\ell}}\,r^{2n-3}dr
	  \simeq
	  y^{b-2(n-1)(\ell-1)}
	  e^{-a\,y^{2\ell}}\quad 	(y\ge 1).
	  \]
The change of variables $r=yt$ shows that 
$I(y)\simeq y^{b+2(n-1)}\,e^{-ay^{2\ell}} J(y)$, where
\[
J(y):= 
\int_{0}^\infty (1+t)^b
e^{-ay^{2\ell}((1+t^2)^{\ell}-1)}\,t^{2n-3}dt.	
\]
We obtain the lower estimate for $I(y)$ by considering the root 
 $t_y>0$ of the equation $y^{2\ell}((1+t^2)^\ell-1)=1$, that is, 
	  \[
	  t_y=\big((1+y^{-2\ell})^{1/\ell}-1\big)^{1/2}
\simeq y^{-\ell},
	  \]
and observing that 
\[	  
J(y)\ge \int_{0}^{t_y} 
(1+t)^b e^{-ay^{2\ell}((1+t^2)^{\ell}-1)}\,t^{2n-3}dt
\simeq \int_{0}^{t_y} t^{2n-3}dt\simeq y^{-2(n-1)\ell}.
\]
	  
In order to get the upper estimate, note that if 
$\ell\ge 1$ then $(1+t^2)^\ell-1\ge \ell t^2$, and so
\[	  
J(y)
\le \int_{0}^{\infty}(1+t)^b e^{-a\ell y^{2\ell}t^2}\,t^{2n-3}dt
\le 2^{\max(b,0)}(J_1(y)+J_2(y)),
\] 
where
\[
J_1(y):= \int_0^1 e^{-a\ell y^{2\ell}t^2}\,t^{2n-3}dt
\quad\text{and}\quad
J_2(y):=\int_1^\infty  e^{-a\ell y^{2\ell}t^2}\,t^{2n-3+b}dt.
	 \]
By making the change of variables $s=y^\ell t$, we have that
\begin{gather*}
J_1(y)=y^{-2(n-1)\ell}\int_0^{y^{\ell}}e^{-a\ell s^2}s^{2n-3}\,ds
\lesssim y^{-2(n-1)\ell}\qquad\mbox{and}\\
J_2(y) = \, y^{-(2n-2+b)\ell}
\int_{y^{\ell}}^{\infty} e^{-a\ell s^2} s^{2n-3+b} ds
\lesssim 
y^{-(2n-2+b)\ell}\int_{y^{\ell}}^{\infty} e^{-a\ell s} ds
\lesssim y^{-2(n-1)\ell},		 
\end{gather*}
which ends the proof. 
	\end{proof}

\begin{lem}\label{lem:estuv}
Let $a>0$ and let $b\in\R$. Then
\[
I(z)=I_{a,b}(z):=\int_\C\frac{e^{-a|v-z|^2}}{(1+|v|)^b}\,dA(v)
\simeq \frac{1}{(1+|z|)^{b}} \qquad(z\in\C)
\]
and
\[
J(z)=J_{a,b}(z):=\int_\C\frac{e^{-a(|v|-|z|)^2}}{(1+|v|)^b}\,dA(v)
\simeq \frac{1}{(1+|z|)^{b-1}} \qquad(z\in\C).
\]
\end{lem}

\begin{proof}
Since $I_{a,b}(z)\simeq I_{1,b}(za^{1/2})$ and 
$J_{a,b}(z)\simeq I_{1,b}(za^{1/2})$, we may assume that $a=1$.
Moreover, 
$I(z)\simeq 1\simeq J(z)$, when $|z|\le1$, so
we only have to prove the estimates for $|z|\ge 1$.
In this case we split each of the above integrals into the
corresponding three integrals on the 
sets~{$S_1=\{v\in\C: |v|<|z|/2\}$,}
$S_2=\{v\in\C: |z|/2\le|v|\le 2|z|\}$
and $S_3=\{v\in\C: |v|> 2|z|\}$, that is, 
$I(z)=I_1(z)+I_2(z)+I_3(z)$ and $J(z)=J_1(z)+J_2(z)+J_3(z)$, 
where
\[
I_k(z):=\int_{S_k}\frac{e^{-|v-z|^2}}{(1+|v|)^b}\,dA(v)
\quad\mbox{and}\quad
J_k(z):=\int_{S_k}\frac{e^{-(|v|-|z|)^2}}{(1+|v|)^b}\,dA(v).
\]

If $v\in S_1$ then $|v-z|\ge|z|-|v|>|z|/2$. Thus
\[
I_1(z)\le J_1(z)\lesssim e^{-|z|^2/4}\int_0^{|z|/2}\frac{r\,dr}{(1+r)^b}
\lesssim e^{-|z|^2/4}(1+|z|)^{|b|+2}
\lesssim\frac1{(1+|z|)^b}.
\]

If $v\in S_2$ then $(1+|z|)/2\le1+|v|\le2(1+|z|)$. Therefore
\[
I_2(z)\simeq\frac1{(1+|z|)^b}\int_{S_2}e^{-|v-z|^2}dA(v)
\,\mbox{ and }\,      
J_2(z)\simeq\frac1{(1+|z|)^b}\int_{S_2}e^{-(|v|-|z|)^2}dA(v).
\]
Since $D(z,1/2)\subset S_2$, we have
\[
0 < \int_{D(0,1/2)}e^{-|w|^2}dA(w)
\le \int_{S_2}e^{-|v-z|^2}dA(v)
\le \int_{\C}e^{-|w|^2}dA(w)<\infty,
\]
and so $I_2(z)\simeq(1+|z|)^{-b}$. On the other hand, $J_2(z)\simeq(1+|z|)^{1-b}$ because
\[
\int_{S_2}e^{-(|v|-|z|)^2}dA(v)
\simeq\int_{|z|/2}^{2|z|}e^{-(r-|z|)^2}r\,dr
\simeq|z|\int_{-|z|/2}^{|z|}e^{-t^2}dt
\simeq|z|.
\]

If $v\in S_3$ then $|v-z|\ge|v|-|z|>|v|/2 $, and hence
\[
I_3(z)\le J_3(z)
\lesssim \int_{2|z|}^\infty \frac{re^{-r^2/4}}{(1+r)^b}\,dr
\le e^{-|z|^2/2}\int_0^\infty \frac{re^{-r^2/8}}{(1+r)^b}\,dr
\lesssim \frac{1}{(1+|z|)^{b}}.\qedhere
\]
\end{proof}

\begin{proof}[Proof of 
 Proposition~{\ref{prop:pnormBergman}}]
Let $p=\infty$. Then the lower estimate follows 
from~{\eqref{eqn:simeq:estimate:H}:}
\begin{align*}
\|K_{\alpha}(\cdot,z)\|_{F^{\infty}_{\alpha,\rho}}
\ge &\, 
K_\alpha(z,z)\,(1+|z|)^\rho\, e^{-\frac\alpha 2|z|^{2\ell}}
= H_{\alpha}(|z|^2)\,
  (1+|z|)^\rho\, e^{-\frac\alpha 2|z|^{2\ell}}\\
\gtrsim &\,(1+|z|^2)^{n(\ell-1)}\,  
            (1+|z|)^\rho\, e^{\frac\alpha 2|z|^{2\ell}}
\simeq  (1+|z|)^{\rho+2n(\ell-1)}\, e^{-\frac\alpha 2|z|^{2\ell}}.
\end{align*}
In order to obtain the upper estimate, first note that \eqref{eqn:rough:estimate:H} and the Cauchy-Schwarz inequality (that is, $|z\overline{w}|\le|z||w|$, for any $z,w\in\C^n$) show that 
\begin{align*}
|K_{\alpha}(w,z)|=|H_{\alpha}(z\overline{w})|
\lesssim &\, (1+|z\overline{w}|)^{n(\ell-1)}\,
         e^{\alpha|z\overline{w}|^{\ell}} \\
\lesssim &\, (1+|z|)^{n(\ell-1)}(1+|w|)^{n(\ell-1)}\,
        e^{\alpha|z|^{\ell}|w|^{\ell}}.
\end{align*}
Therefore 
$\|K_{\alpha}(\cdot,z)\|_{F^{\infty}_{\alpha,\rho}}
\lesssim 
   (1+|z|)^{n(\ell-1)}\,e^{\frac{\alpha}2|z|^{2\ell}} M(|z|)$,
where
\[
M(|z|)
=\sup_{w\in\C}\,(1+|w|)^{\rho+n(\ell-1)}
	    e^{-\frac{\alpha}2(|w|^{\ell}-|z|^\ell)^2}
\simeq\sup_{x\ge0}\,(1+x)^{\frac{\rho+n(\ell-1)}{\ell}}	       
              e^{-\frac{\alpha}2(x-|z|^\ell)^2}.
\]
Since, by Lemma~{\ref{lem:est:sup}}, 
$M(|z|) \simeq (1+|z|^{\ell})^{\frac{\rho+n(\ell-1)}{\ell}}
        \simeq (1+|z|)^{\rho+n(\ell-1)}$,
we have just proved the upper estimate in this case.
	
Now assume that $p<\infty$.	
By making the change of variables $u=Uw$, where $U:\C^n\to\C^n$ 
is a unitary mapping such that $U(z)=(|z|,0\,\dots,0)$,
 we get that
\[
\|K_{\alpha}(\cdot,z)\|^p_{F^{p}_{\alpha,\rho}}
\simeq
\int_{\C} |H_{\alpha}(|z|u_1)|^p\,\Psi(u_1)\,dA(u_1),
\]
where
\[
\Psi(u_1):=	\int_{\C^{n-1}} (1+|u_1|+|u'|)^{\rho p}\,
e^{-\frac{\alpha p}{2}(|u_1|^2+|u'|^2)^{\ell}}\,dV_{n-1}(u').
\]
Then Lemma \ref{lem:est:Cm} implies that
\begin{equation}\label{eqn:bergman:kernel:norm:estimate}\begin{split}
&\|K_{\alpha}(\cdot,z)\|^p_{F^{p}_{\alpha,\rho}}\\ &\quad\simeq
\int_{\C} |H_{\alpha}(|z|u_1)|^p\, 
(1+|u_1|)^{\rho p-2(n-1)(\ell-1)}\,
e^{-\frac{\alpha p}2\,|u_1|^{2\ell}}\,dA(u_1).
\end{split}
\end{equation}
Now pick $N>2$ satisfying the statement of 
Corollary~{\ref{cor:pointwB}}. 
Then note that~{\eqref{eqn:rough:estimate:H}}
implies 
\[
\mathcal{X}_{S_N}(u_1)
\lesssim |H_{\alpha}(|z|u_1)|^p  
\lesssim (1+|u_1|)^{np(\ell-1)} e^{\alpha p2^{\ell}|u_1|^{\ell}}
\quad(|z|\le2,\,u_1\in\C).
\]
Thus~{\eqref{eqn:bergman:kernel:norm:estimate}} shows that
\[
\|K_{\alpha}(\cdot,z)\|^p_{F^{p}_{\alpha,\rho}}\simeq 1
\simeq (1+|z|)^{\rho+2n(\ell-1)/p'}e^{\frac{\alpha}2|z|^{2\ell}}
\quad(|z|\le2), 
\]
so we only have to prove the norm estimate for $|z|>2$.
In order to do that, we split the integral 
in~{\eqref{eqn:bergman:kernel:norm:estimate}} as the sum of the 
three integrals $\mathcal{I}_1(|z|)$, $\mathcal{I}_2(|z|)$ and 
$\mathcal{I}_3(|z|)$ on the sets
 $E_1=\{u_1\in\C:|u_1|>1, |\arg u_1|\le\pi/(N\ell)\}$, 
 $E_2=\{u_1\in\C:|u_1|>1, |\arg u_1|>\pi/(N\ell)\}$ and 
$E_3=\{u_1\in\C:|u_1|\le 1\}$,  respectively.
	
To estimate $\mathcal{I}_1(|z|)$ recall 
that~{\eqref{eqn:simeq:estimate:H}} gives 
\[
|H_{\alpha}(|z|u_1)|^p \simeq 
(|z||u_1|)^{np(\ell-1)} e^{\alpha p|z|^{\ell}\Re u_1^{\ell}}
\quad(u_1\in E_1,\,|z|>2),
\]
so 
\[
\mathcal{I}_1(|z|)\simeq 
|z|^{np(\ell-1)}e^{\frac{\alpha p}2|z|^{2\ell}}
\int_{E_1}|u_1|^{np(\ell-1)+\rho p-2(n-1)(\ell-1)} e^{-\frac{\alpha p}2|u_1^{\ell}-|z|^{\ell}|^2}dA(u_1).
\]
By making the change of variables $v=u_1^\ell$ we have that
\[
\mathcal{I}_1(|z|)
\simeq |z|^{np(\ell-1)}e^{\frac{\alpha p}2|z|^{2\ell}}
	   \int_{\{|v|\ge 1, |\arg v|\le \pi/N\}}
	   |v|^{\beta} e^{-\frac{\alpha p}2|v-|z|^{\ell}|^2} dA(v),
\]
where $\beta:=(n(\ell-1)(p-2)+\rho p)/\ell$.
Since for $|z|>2$ we have the inclusions
\begin{align*}
D(|z|^{\ell},\sin(\pi/N))
&\subset
\{v\in\C: |v|>1\}\cap D(|z|^{\ell},|z|^{\ell}\sin(\pi/N))\\
&\subset\{v\in\C: |v|>1, |\arg v|\le\pi/N\},
\end{align*}
the preceding integral $\mathcal{I}'_1(|z|)$ satisfies
\[
\mathcal{I}'_1(|z|)
\ge\int_{D(|z|^{\ell},\sin(\pi/N))}
	   |v|^{\beta} e^{-\frac{\alpha p}2|v-|z|^{\ell}|^2} dA(v)
\simeq|z|^{\beta\ell}.
\]
Moreover, Lemma~{\ref{lem:estuv}} shows that
$\mathcal{I}'_1(|z|)\lesssim I_{\alpha p/2,-\beta}(|z|^{\ell})
\simeq |z|^{\beta\ell}$. 
It follows that $\mathcal{I}'_1(|z|)\simeq |z|^{\beta\ell}
=|z|^{n(\ell-1)(p-2)+\rho p}$, and hence
\begin{equation}\label{eqn:estimate1}
\mathcal{I}_1(|z|)\simeq 
|z|^{np(\ell-1)}e^{\frac{\alpha p}2|z|^{2\ell}}\mathcal{I}'_1(|z|)
\simeq (1+|z|)^{2n(\ell-1)(p-1)+\rho p}\,
       e^{\frac{\alpha p}2|z|^{2\ell}}.
\end{equation}
Now we estimate $\mathcal{I}_2(|z|)$.
 By~{\eqref{eqn:lesssim:estimate:H}},
\[
|H_{\alpha}(|z|u_1)|^p \lesssim
(|z||u_1|)^{np(\ell-1)} e^{\alpha p\cos(\frac{\pi}N)\,|z|^{\ell} |u_1|^{\ell}}
\quad(u_1\in E_2,\,|z|>2),
\]
so $\mathcal{I}_2(|z|)\lesssim 
|z|^{np(\ell-1)}
e^{\frac{\alpha p}2\cos^2(\frac{\pi}N)\,|z|^{2\ell}}\mathcal{I}'_2(|z|)$, where
\begin{align*}
\mathcal{I}'_2(|z|)
:= & \int_{E_2}|u_1|^{np(\ell-1)+\rho p-2(n-1)(\ell-1)}     
    e^{-\frac{\alpha p}2
    \{|u_1|^{\ell}-|z|^{\ell}\cos(\frac{\pi}N)\}^2} dA(u_1)\\
\simeq & \int_1^{\infty} 
         r^{1+np(\ell-1)+\rho p-2(n-1)(\ell-1)}     
    e^{-\frac{\alpha p}2
    \{r^{\ell}-|z|^{\ell}\cos(\frac{\pi}N)\}^2} dr.
\end{align*}
Then we make the change of variables $t=r^{\ell}$ to get that
\[
\mathcal{I}'_2(|z|)\simeq 
\int_1^{\infty}t^{\beta+1}e^{-\frac{\alpha p}2
    \{t-|z|^{\ell}\cos(\frac{\pi}N)\}^2} dt,
\]
so Lemma~{\ref{lem:estuv}} shows that 
$\mathcal{I}'_2(|z|)\lesssim 
  J_{\frac{\alpha p}2,-\beta}(|z|^{\ell}\cos(\frac{\pi}N))
\simeq|z|^{\beta\ell+\ell}$.
 Hence
\begin{equation}\label{eqn:estimate2}
\mathcal{I}_2(|z|)\lesssim 
|z|^{np(\ell-1)+\beta\ell+\ell}
e^{\frac{\alpha p}2\cos^2(\frac{\pi}N)\,|z|^{2\ell}}\lesssim
(1+|z|)^{2n(\ell-1)(p-1)+\rho p}\,
       e^{\frac{\alpha p}2|z|^{2\ell}}.
\end{equation}
Finally, since by~{\eqref{eqn:rough:estimate:H}} we have that
\[
|H_{\alpha}(|z|u_1)|^p\lesssim (1+|z|)^{np(\ell-1)}e^{\alpha p|z|^{\ell}}\quad(u_1\in E_3, |z|>2),
\]
we obtain that
\begin{equation}\label{eqn:estimate3}
\mathcal{I}_3(|z|)
\lesssim (1+|z|)^{n(\ell-1)}e^{\alpha|z|^{\ell}}		
\lesssim (1+|z|)^{2n(\ell-1)(p-1)+\rho p}e^{\frac{\alpha p}2|z|^{2\ell}}.
\end{equation}
Taking into account~{\eqref{eqn:estimate1}}, 
\eqref{eqn:estimate2} and~{\eqref{eqn:estimate3}}, we conclude
that
\[
\|K_{\alpha}(\cdot,z)\|^p_{F^{p}_{\alpha,\rho}}
\simeq (1+|z|)^{2n(\ell-1)(p-1)+\rho p}
       e^{\frac{\alpha p}2|z|^{2\ell}}    \quad(|z|>2),
\] 
which ends the proof.	
\end{proof}

\begin{cor}\label{cor:qnormBergman}
	Let $1\le p\le\infty$. Then
	\[
	\|K_{\alpha}(\cdot,z)\|_{F^{p}_{\beta,\rho}}
	\simeq (1+|z|)^{\rho+2n(\ell-1)/p'}
	e^{\frac{\alpha^2}{2\beta}|z|^{2\ell}}\quad(z\in\C^n).
	\]
\end{cor}

\begin{proof}
	Since 
	$K_{\alpha}(\delta z, w)=\delta^{-n}K_{\delta^\ell \alpha}(z,w)$,
	for $\delta=(\beta/\alpha)^{1/\ell}$, we have
	\begin{align*}
	\|K_{\alpha}(\cdot,z)\|_{F^{p}_{\beta,\rho}}
	&\simeq 	\|K_{\beta}(\cdot,z/\delta)\|_{F^{p}_{\beta,\rho}}\\
	&\simeq (1+|z|/\delta)^{\rho+2n(\ell-1)/p'}
	e^{\frac{\beta}2|z/\delta|^{2\ell}}\\
	&\simeq(1+|z|)^{\rho+2n(\ell-1)/p'}
	e^{\frac{\alpha^2}{2\beta}|z|^{2\ell}}.
	\end{align*}
	This ends the proof.
	\end{proof}

\subsection{The Bergman projection}\quad\par

The next lemma shows that the Bergman 
 projection $P_{\alpha}$ 
is pointwise well defined on $L^{p}_{\beta,\rho}$  
if and only if $\beta<2\alpha$.

\begin{lem}\label{lem:well-defined}
Let $\z\in \C^n$ and assume $1\le p\le \infty$.  

\begin{enumerate}
\item If for $\z\ne 0$ the form $U_\z:L^{2}_{\alpha}\to\C$, defined by
$U_\z(f)=P_\alpha(f)(\z)$, 	
is bounded on the normed space
 $(L^{2}_{\alpha}\cap L^{p}_{\beta,\rho},
\,\|\cdot\|_{L^{p}_{\beta,\rho}})$	
then  $\beta<2\alpha$.

\item Conversely, if $\beta<2\alpha$ then 
$U_\z:L^{p}_{\beta,\rho}\to\C$, defined by
\[
U_\z(f)=\int_{\C^n}f(w)K_{\alpha}(\z,w)e^{-\alpha |w|^{2\ell}}dV(w),
\]
 is bounded and 
	\[
	\|U_\z\|\lesssim (1+|\z|)^{-\rho+2n(\ell-1)/p}
	e^{\frac 12\frac{\alpha^2}{2\alpha-\beta}|\z|^{2\ell}}.
	\]
	\end{enumerate}
	\end{lem}

\begin{proof}
	Assume that $U_\z$ is bounded on  $(L^{2}_{\alpha}\cap L^{p}_{\beta,\rho},\,\|\cdot\|_{L^{p}_{\beta,\rho}})$. Then, by Hahn-Banach theorem's, $U_\z$ extends to a bounded operator on $L^{p}_{\beta,\rho}$, which we also denote by $U_\z$.
	
	Let $\nu$ be a multi-index.
	It is clear that the function
	\begin{equation}\label{eqn:f1}
	f(z):=\frac{z^\nu}{(1+|z|)^{|\nu|+\rho+2n+1}}
	\, e^{\frac{\beta}2|z|^{2\ell}}
	\end{equation}
	belongs to $L^{p}_{\beta,\rho}$.
	Let $\mathcal{X}_R$ be the characteristic function of the open ball $B_R$ centered at $0$ with radius $R$.
	Then the function $f_R=\mathcal{X}_R\cdot f$ is in $ L^{2}_{\alpha}\cap L^{p}_{\beta,\rho}$ and 
	$\|f_R-f\|_{F^{p}_{\beta,\rho}}\to 0$ as $R\to\infty$.  
	Since 
	\[
	K_{\alpha,z}(w)
	=\sum_{\mu\in \N^n}\frac{w^\mu\overline{z}^\mu}{\|w^\mu\|^2_{F^{2}_{\alpha}}},
	\]
	where the series converges in $L^{2}_\alpha$, 
\[
P_\alpha(f_R)(z)=\langle f_R,K_{\alpha,z}\rangle_\alpha
=\sum_{\mu\in \N^n}\frac{z^\mu}{\|w^\mu\|^2_{F^{2}_{\alpha}}
}\,\langle f_R,w^\mu\rangle_\alpha.
\]
By integration in polar coordinates we have 
$\langle f_R,w^\mu\rangle_\alpha^\ell=\delta_{\mu,\nu} c_{\nu}(R)$, where 
\[
c_{\nu}(R):=
\int_{B_R} \frac{|w^\nu|^2}{(1+|w|)^{|\nu|+\rho+2n+1}}
\, e^{(\frac{\beta}2-\alpha)|w|^{2\ell}}\,dV(w).
	\]
	Thus $U_\z(f_R)=P_\alpha(f_R)(\z)	=c_{\nu}(R)\,
	\z^\nu/\|w^\nu\|^2_{F^{2}_{\alpha}}$.
	So, by the hypothesis and the monotone convergence theorem,
	\[
	U_\z(f)=\lim_{R\to\infty}U_\z(f_R)=	\frac{\z^\nu}{\|w^\nu\|^2_{F^{2}_{\alpha}}}
\int_{\C^n} \frac{|w^\nu|^2}{(1+|w|)^{|\nu|+\rho+2n+1}}
\, e^{(\frac{\beta}2-\alpha)|w|^{2\ell}}\,dV(w).
	\]
It follows that for any  $\nu$  such that 
	$\z^\nu\ne 0$  we have that the above integral is finite. Choosing $\nu$ such that 
	$
	|\nu|\ge 1+\rho
	$
	we obtain that $\beta<2\alpha$.
	
	Next assume $\beta<2\alpha$. 
	Let	$F_\z(w):=G(w)H_\z(w)$, where
	\begin{align*}
	G(w)
	&:=|f(w)|
	(1+|w|)^{\rho}e^{-\frac\beta 2|w|^{2\ell}}
	\quad\text{and}\quad\\
	H_\z(w)&:=
	|K_{\alpha}(\z,w)| (1+|w|)^{-\rho}
	e^{-(\alpha-\frac\beta 2)|w|^{2\ell}}.
	\end{align*}

 Since $\|G\|_{L^p}= \|f\|_{L^{p}_{\beta,\rho}}$, we obtain 
	\[
|U_\z(f)|\le\|F_\z\|_{L^1}\le\|K_\alpha(\cdot,\z)\|_{L^{p'}_{2\alpha-\beta,-\rho}}  \|f\|_{L^{p}_{\beta,\rho}}.
	\]
	Hence Corollary \ref{cor:qnormBergman} ends the proof.
	\end{proof}

\begin{rem}
	From the pointwise estimate of $|K_\alpha(z,w)|$ with  $z\overline w\in S^{\delta}_N$, given in Proposition \ref{prop:pointK},
	it is easy to check that if $\beta\ge 2\alpha$ and $f$ is the function defined  in \eqref{eqn:f1} with $\nu=0$, 
	then $F_\z\notin L^1$. So $U_\zeta(f)$ is not well defined.
	\end{rem}

\begin{cor}\label{cor:embedFinfty}
	Let $1\le p< \infty$. Then 
	\[
	F^{p}_{\alpha,\rho}\hookrightarrow F^{\infty}_{\alpha,\rho-2n(\ell-1)/p},
	\]
	that is,
	\[
	|f(z)|\lesssim
	\|f\|_{F^{p}_{\alpha,\rho}}(1+|z|)^{-\rho+(2n(\ell-1))/p}
	e^{\alpha |z|^{2\ell}/2}\quad(f\in F^{p}_{\alpha,\rho},\,z\in\C^n).
	\]
\end{cor}

\begin{cor}\label{cor:representation}
	Let $1\le p\le\infty$ and let $\beta<2\alpha$. If $f\in F^{p}_{\beta,\rho}$ then $f=P_{\alpha}f$.
\end{cor}

\begin{proof}
	If $p<\infty$, 
	the space $F^{2}_\alpha\cap F^{p}_{\beta,\rho}$ is dense in $F^{p}_{\beta,\rho}$ and $P_\alpha$ is the identity on $F^{2}_\alpha$, so $P_\alpha$ is the identity on $F^{p}_{\beta,\rho}$.
	
	The case $p=\infty$ follows from the previous one
	 by noting that $F^{\infty}_{\beta,\rho} \subset
	  F^{p}_{\beta',\rho}$,
	for any $\beta'\in (\beta,2\alpha)$. 
\end{proof}

\begin{prop}\label{prop:PLpontoFp}
For $1\le p\le\infty$ the Bergman operator 
	 $P_{\alpha}$ is a bounded projection from $L^{p}_{\alpha,\rho}$ onto $F^{p}_{\alpha,\rho}$.
	\end{prop}

\begin{proof}
First we consider the case $1< p<\infty$.
	By Proposition \ref{prop:pnormBergman}, the function 
			\[
		\Omega_\alpha(z,w):=
		e^{-\frac{\alpha}2 |z|^{2\ell}}
		|K_\alpha(z,w)|e^{-\frac{\alpha}2 |w|^{2\ell}}
		\]
		satisfies
	\begin{equation}\label{eqn:omega}
		\int_{\C^n}\Omega_\alpha(z,w) (1+|w|)^c dV(w)
		\simeq \|K_\alpha(\cdot,z)\|_{L^{1}_{\alpha,c}}\simeq (1+|z|)^c.
	\end{equation}
	If $\varphi\in L\pla$, then H\"older's inequality and \eqref{eqn:omega} with $c=0$ give 
	\begin{equation}\label{eqn:PLpontoFp}
e^{-\frac{p\alpha}2|z|^{2\ell}} |P_{\alpha}(\varphi)(z)|^p
\lesssim \int_{\C^n} |\varphi(w)|^p \Omega_\alpha(z,w) e^{-\frac{p\alpha}2|w|^{2\ell}} dV(w).
\end{equation}
So Fubini's theorem and \eqref{eqn:omega} with $c=\rho\, p$ imply
$\|P_{\alpha}(\varphi)\|_{L^{p}_{\alpha,\rho}}\lesssim  \|f\|_{L^{p}_{\alpha,\rho}}$.

If $p=1$ then \eqref{eqn:PLpontoFp} is obvious and, as in the above case, we obtain the result.

If $p=\infty$ then
	\begin{equation*}
(1+|z|)^\rho e^{-\frac{\alpha}2|z|^{2\ell}} |P_{\alpha}(\varphi)(z)|
\lesssim \|f\|_{L^{\infty}_{\alpha,\rho}}
(1+|z|)^\rho
\int_{\C^n} \frac{\Omega_\alpha(z,w) }{(1+|w|)^{\rho}} dV(w).
\end{equation*}
So \eqref{eqn:omega} shows that 
$\|P_{\alpha}(\varphi)\|_{L^{\infty}_{\alpha,\rho}}\lesssim  \|f\|_{L^{\infty}_{\alpha,\rho}}$.
\end{proof}

\begin{cor}\label{cor:dualF}
Let $1\le p<\infty$.
	Then the dual of $F^{p}_{\alpha,\rho}$, 	
	with respect to the pairing
	 $\langle\cdot,\cdot\rangle_\alpha$, is
	  $F^{p'}_{\alpha,-\rho}$.
	\end{cor}

\begin{proof}
	From the classical $L^p$-duality it is easy to check that  the dual of $L^{p}_{\alpha,\rho}$, with respect to the pairing $\langle\cdot,\cdot\rangle_\alpha^\ell$, is $L^{p'}_{\alpha,-\rho}$.
		This result together with  Proposition \ref{prop:PLpontoFp} prove the corollary.
	\end{proof}

\section{Proof of Theorem \ref{thm:embeddings}}
\label{sec:embeddings}

The case $\ell=1$ and $\rho=\eta=0$ is well
 known (see \cite{janson-peetre-rochberg}). 
 For $n=1$, the theorem can be deduced from the
  characterization of Carleson measures obtained
   in  \cite[Theorem 1]{constantin-pelaez}.

\subsection{Necessary conditions for all $p$ and $q$} 

\begin{lem}\label{lem:necessarypq}
If 
$F^{p}_{\beta,\rho}\hookrightarrow F^{q}_{\gamma,\eta}$,
 then either $\beta<\gamma$ or $\beta=\gamma$ and 
\[
2n(\ell-1)\bigl(\tfrac 1p-\tfrac 1q\bigr)\le\rho -\eta.
\]
\end{lem}

\begin{proof}
By Corollary \ref{cor:qnormBergman} the ratio
\[
\frac{\|K_{\alpha}(\cdot,z)\|_{F^{q}_{\gamma,\eta}}}{\|K_{\alpha}(\cdot,z)\|_{F^{p}_{\beta,\rho}}}
\simeq\frac{ (1+|z|)^{\eta+2n(\ell-1)/q'}
	e^{\frac{\alpha^2}{2\gamma}|z|^{2\ell}}}
{ (1+|z|)^{\rho+2n(\ell-1)/p'}e^{\frac{\alpha^2}{2\beta}|z|^{2\ell}}}
\]
is bounded  if and only if $\beta$, $\gamma$, $\rho$ and $\eta$ satisfy the above conditions.
\end{proof}

\subsection{Proof of Theorem \ref{thm:embeddings} for $1\le p\le q\le\infty$}\quad\par

The  next lemma shows that the necessary conditions obtained in the above section are also sufficient, which proves  Theorem \ref{thm:embeddings} for $1\le p\le q\le\infty$.
 
\begin{lem}\label{lem:condsuffpleq}
If  either $\beta<\gamma$ or $\beta=\gamma$ and 
\[
2n(\ell-1)\bigl(\tfrac 1p-\tfrac 1q\bigr)\le\rho -\eta,
\]
then
$F^{p}_{\beta,\rho}\hookrightarrow F^{q}_{\gamma,\eta}$, 
provided that $1\le p\le q\le\infty$.
\end{lem}

\begin{proof}	
If $p=q$ then $\eta\le\rho$. Hence
$(1+|z|)^\eta e^{-\frac{\gamma}2|z|^2{2\ell}}\lesssim 
(1+|z|)^\rho e^{-\frac{\beta}2|z|^2{2\ell}}$
which proves the embedding $F^{p}_{\beta,\rho}\hookrightarrow F^{p}_{\gamma,\eta}$.

The case  $p<q=\infty$ is a consequence of 
Corollary 	\ref{cor:embedFinfty} and the case $p=q$. Indeed, $F^{p}_{\beta,\rho}\hookrightarrow F^{\infty}_{\beta,\rho-2n(\ell-1)/p}\hookrightarrow
F^{\infty}_{\gamma,\eta}.
$

Assume $1\le p<q<\infty$ and let  $f\in F^{p}_{\beta,\rho}$. Consider $F$ the function defined by
\[
F(z):=
|f(z)|(1+|z|)^{\eta }e^{-\frac{\gamma }2|z|^{2\ell}}
=G(z)^{p/q}H(z)^{(q-p)/q},
\]
where
\[
G(z):=|f(z)|(1+|z|)^{\rho }e^{-\frac{\beta }2|z|^{2\ell}}
\]
and
\[
H(z):=
|f(z)|(1+|z|)^{\frac{\eta q -\rho p}{q-p}}
e^{-\frac{\gamma q-\beta p}{2(q-p)}|z|^{2\ell}}.
\]
By Corollary 	\ref{cor:embedFinfty} and the hypotheses on  $\rho$ and $\eta$,  we have 
\begin{align*}
|H(z)|&\lesssim \|f\|_{F^{p}_{\beta,\rho}}
(1+|z|)^{\frac{\eta q -\rho p}{q-p}-\rho+\frac{2n(\ell-1)}{p}}
e^{\bigl(-\frac{\gamma q-\beta p}{2(q-p)}+\frac{\beta}2\bigr)|z|^{2\ell}}\\
&= \|f\|_{F^{p}_{\beta,\rho}}
(1+|z|)^{(\eta -\rho )\frac{q}{q-p}+\frac{2n(\ell-1)}{p}}
e^{-\frac{(\gamma -\beta)q }{2(q-p)}|z|^{2\ell}}\\
&\lesssim  \|f\|_{F^{p}_{\beta,\rho}}.
\end{align*}
Hence
\[
\|f\|_{F^{q}_{\gamma,\eta}}^q=\|F\|_{L^q}^q\lesssim \|f\|_{F^{p}_{\beta,\rho}}^{q-p}\,\|G\|_{L^p}^p=\|f\|_{F^{p}_{\beta,\rho}}^{q}.\qedhere
\]
	\end{proof}
	
Observe that, for $1\le p\le q\le \infty$, by Lemmas \ref{lem:necessarypq} and \ref{lem:condsuffpleq}, the fact that the embedding 
$F^{p}_{\beta,\rho}\hookrightarrow F^{q}_{\gamma,\eta}$ holds is only a question 
 of  growth, that is, $F^{p}_{\beta,\rho}\hookrightarrow F^{q}_{\gamma,\eta}$ if and only if 
$F^{\infty}_{\beta,\rho-2n(\ell-1)/p}\hookrightarrow 
	F^{\infty}_{\gamma,\eta-2n(\ell-1)/q}$.

\subsection{Sufficient conditions for $1\le q<p\le\infty$}

\begin{lem}\label{lem:condsuffip>q}
	If  either $\beta<\gamma$ or $\beta=\gamma$ and 
	$
2n\left(\frac 1q-\frac 1p\right)< \rho-\eta
$,
then we have
$F^{p}_{\beta,\rho}\hookrightarrow 
 F^{q}_{\gamma,\eta}$, provided that  $1\le q< p\le\infty$. 
\end{lem}

\begin{proof}
	Let $f\in F^{p}_{\beta,\rho}$.  Assume first $p=\infty$. In this case $q(\rho-\eta)>2n$, so  the hypotheses on the parameters give
	\begin{align*}
	\|f\|_{F^{q}_{\gamma,\eta}}^q
	&=
	\int_{\C^n}|f(z)|^q(1+|z|)^{\eta q}
	e^{-\frac{\gamma q }{2}|z|^{2\ell}}dV(z)\\
&\lesssim \|f\|_{F^{\infty}_{\beta,\rho}}^q
	\int_{\C^n}(1+|z|)^{-(\rho-\eta)q}
	e^{-\frac{(\gamma-\beta)q }{2}|z|^{2\ell}}dV(z)
	\lesssim \|f\|_{F^{\infty}_{\beta,\rho}}^q.
	\end{align*}
	
	Next assume $p$ finite. In this case 
$(\rho-\eta)\frac{pq}{p-q}>2n$.
 Consider the function 
	\[
	F(z):=
	|f(z)|(1+|z|)^{\eta }e^{-\frac{\gamma }2|z|^{2\ell}}
	=G(z)H(z),
	\]
	where
	\[
	G(z):=|f(z)|(1+|z|)^{\rho }e^{-\frac{\beta }2|z|^{2\ell}}
	\quad\text{and}\quad
	H(z):=
	(1+|z|)^{(\eta  -\rho )}e^{-\frac{\gamma -\beta }{2}|z|^{2\ell}}.
	\]
	By H\"older's inequality with exponent $p/q>1$ we have 
\begin{align*}
\|f\|_{F^{q}_{\gamma,\eta}}&=\|F\|_{L^q}\le
\|G\|_{L^p}\|H\|_{L^{pq/(p-q)}}\\
&= \|f\|_{F^{p}_{\beta,\rho}}
	\left(\int_{\C^n}(1+|z|)^{-(\rho-\eta)\frac{pq}{p-q}}
e^{-\frac{\gamma-\beta}2\frac{pq}{p-q}|z|^{2\ell}}dV(z)\right)^{\frac{p-q}{pq}}.
\end{align*}
Therefore $\|f\|_{F^{q}_{\gamma,\eta}}
\lesssim \|f\|_{F^{p}_{\beta,\rho}}$.
	\end{proof}

\subsection{Necessary conditions for $1\le q<p<\infty$ and  $\beta=\gamma$}\quad\par

\begin{prop}\label{prop:necpgrq}
If $1\le q<p<\infty$ and  
$F^{p}_{\beta,\rho}\hookrightarrow F^{q}_{\beta,\eta}$ 
then
$
2n\left(\frac 1q-\frac 1p\right)< \rho-\eta.
$
\end{prop}

The proof of Proposition \ref{prop:necpgrq} follows from  the ideas in \cite{luecking}. We need some technical results.

For $r>0$, let $\tau_r:\C\to (0,\infty)$ be the function defined by
\begin{equation}\label{eqn:tau}
\tau_r(z):=r(1+|z|)^{1-\ell}
\end{equation}
and let $B_r(z):=B(z,\tau_r(z))$.

Note that   $\tau_r$
is a radius function in the sense of \cite[p.1617-1618]{dallara}, that is,
\begin{equation}\label{eqn:radiusf}
1+|z|\simeq 1+|w|\quad (z\in\C^n,\,\,w\in B_r(z)).
\end{equation}
Then we have:

\begin{lem}[{\cite[Proposition 7]{dallara}}]\label{lem:cover}
	For any $r>0$ there exists a sequence $\{z_k\} $ in $\C^n$ 
	such that the Euclidean  balls $B_k:=B_r(z_k)$ 
	 satisfy:
	\begin{enumerate}
		\item $\cup_k B_k=\C^n$.
		\item The overlapping of  the balls $B_k$ is finite,
		 that is, there exists $N_r\in\N$ such that 
		$\sum_k \mathcal{X}_{B_k}(z)\le N_r$ for any $z\in\C^n$.
		\end{enumerate}
	\end{lem}

The following lemma states a subharmonic type estimate.

\begin{lem}\label{lem:propertiescover}\quad\par

	\begin{enumerate}
		\item \label{item:propertiescover1}
There exists $r>0$ such that 
		\[
		 |K_\alpha(z,w)|e^{-\frac{\alpha}2|w|^{2\ell}}
		e^{-\frac{\alpha}2|z|^{2\ell}}
	\simeq (1+|z|)^{2n(\ell-1)} 
		\quad (w\in B_r(z)).
		\]
	
		\item \label{item:propertiescover2} Let $1\le p<\infty$, $\rho\in\R$ and $r>0$. There exists $C=C_{\alpha,p,\rho,r}>0$ such that 
		\[
		|f(z)|^p(1+|z|)^{\rho p-2n(\ell-1)}e^{-\frac{\alpha p}2|z|^{2\ell}}
		\le C\int_{B_r(z)}	|f(w)|^p(1+|w|)^{\rho p}e^{-\frac{\alpha p}2|w|^{2\ell}}\,dV(w),
		\]
for any $z\in\C^n$.
	\end{enumerate}
	\end{lem}

\begin{proof}
We begin proving \eqref{item:propertiescover1}.
By Remark \ref{rem:pointwB}, we may assume that 
$z=(|z|,0,\cdots,0)$. Then we have to prove that 
	\begin{equation}\label{eqn:propertiescover1}
		 |H_\alpha(|z|w_1)|e^{-\frac{\alpha}2|w|^{2\ell}}
		e^{-\frac{\alpha}2|z|^{2\ell}}
	\simeq (1+|z|)^{2n(\ell-1)}
		\quad (w\in B_r(z)).
		\end{equation}

By Corollary \ref{cor:pointwB}, there exist $\delta>0$ and $N>2$ satisfying \eqref{eqn:simeq:estimate:H}.
For $r>0$ small enough we have
$
|z|w_1\in S^\delta_N 
$,
 for any $z\in\C^n$ and $w\in B_r(z)$.
By \eqref{eqn:simeq:estimate:H},
\begin{equation}\label{eqn:propertiescover2}
 |H_\alpha(|z|w_1)|\simeq 	(1+|z||w_1|)^{n(\ell-1)}e^{\alpha|z|^\ell\Re w_1^\ell}	\quad (w\in B_r(z)).
\end{equation}

In particular for $|z|\le 2r$ the terms in 
 \eqref{eqn:propertiescover1} 
are comparable to a positive constant and there is nothing to prove.

Now assume $|z|> 2r$. 
In this case,  $|w_1|\simeq  |z|$ for $w\in B_r(z)$. Hence, by \eqref{eqn:propertiescover2}, the equivalence  \eqref{eqn:propertiescover1} will be a consequence of 
\begin{equation}\label{eqn:propertiescover3}
e^{\alpha|z|^\ell\Re w_1^\ell}e^{-\frac{\alpha}2|w|^{2\ell}}
		e^{-\frac{\alpha}2|z|^{2\ell}}\simeq 1	\quad (w\in B_r(z)).
\end{equation}
First note that 
\begin{align*}
e^{\alpha|z|^\ell\Re w_1^\ell}e^{-\frac{\alpha}2|w|^{2\ell}}
		e^{-\frac{\alpha}2|z|^{2\ell}}
&=e^{\alpha|z|^\ell\Re w_1^\ell}
e^{-\frac{\alpha}2(|w_1|^2+|w'|^2)^{\ell}}
		e^{-\frac{\alpha}2|z|^{2\ell}}\\
&=e^{-\frac{\alpha}2||z|^\ell-w_1^\ell|^2}
e^{-\frac{\alpha}2[(|w_1|^2+|w'|^2)^\ell-|w_1|^{2\ell}]}.
\end{align*}
 By mean value theorem, for $w\in B_r(z)$ we have 
\[
0\le ||z|^\ell-w_1^\ell|\lesssim (|z|+r(1+|z|)^{1-\ell})^{\ell-1}(1+|z|)^{1-\ell}\simeq 1
\]
and 
\[
(|w_1|^2+|w'|^2)^\ell-|w_1|^{2\ell}\lesssim 
(|w_1|^2+|w'|^2)^{\ell-1}|w'|^2\lesssim 
|z|^{2(\ell-1)}(1+|z|)^{2(1-\ell)}\simeq 1,
\]
we obtain \eqref{eqn:propertiescover3}.

	In order to prove part \eqref{item:propertiescover2}, note that, by \eqref{eqn:radiusf}, the case $\rho\ne 0$  follows from the result  for $\rho=0$.   
This last case can be deduced using the arguments in the proofs of Proposition 12 and of Lemma 13 in  \cite{dallara}.

Let $\varphi$ be a real  $\mathcal{C}^2$-function
on the closed unit ball $\overline{B(0,1)}$ of $\C^n$.
It is well known (see for instance \cite{andersson})  that there exists a real $\mathcal{C}^2$-function $\psi$ on $B(0,1)$
such that 
\[
\partial\overline{\partial}\psi=\partial\overline{\partial}\varphi\quad\text{and}\quad \|\psi\|_{L^\infty(B(0,1))}\le C\|\partial\overline{\partial}\varphi\|
_{L^\infty(B(0,1))}.
\]
By rescaling, we get that if $\varphi$ is a real  $\mathcal{C}^2$-function
on the closed  ball $\overline{B(z,R)}$, then  there is a real $\mathcal{C}^2$-function $\psi$ on $B(z,R)$
such that 
\[
\partial\overline{\partial}\psi=\partial\overline{\partial}\varphi\quad\text{and}\quad \|\psi\|_{L^\infty(B(z,R))}\le CR^2\|\partial\overline{\partial}\varphi\|_{L^\infty(B(z,R))}.
\]
Applying this result to the function 
 $\varphi(w)=|w|^{2\ell}$ and to the ball $B_r(z)$  there exists  a real  $\mathcal{C}^2$-function 
$\psi_z$ on $B_r(z)$ such that 
 $\partial\overline{\partial}\psi_z=\partial\overline{\partial}\varphi$ 
and, by \eqref{eqn:radiusf},
\[
\|\psi_z\|_{L^\infty(B_r(z))}
\le C r^2 (1+|z|)^{2(1-\ell)}\sup_{w\in B_r(z)}|w|^{2(\ell-1)}\le C'\,r^2.
\]
Since 
$\psi_z-\varphi$ is a pluriharmonic function on $B_r(z)$, it is the real part of a holomorphic function $h_z$ on $B_r(z)$. 
Thus we have
\begin{align*}
|f(z)|^pe^{-\frac{\alpha p}2|z|^{2\ell}}
&\simeq |f(z)e^{\frac{\alpha}2 h_z(z)}|^p
\le\frac 1{|B_r(z)|}
\int_{B_r(z)}|f(w)e^{\frac{\alpha}2 h_z(w)}|^p\,dV(w)\\
&\simeq (1+|z|)^{2n(\ell-1)}
\int_{B_r(z)}|f(w)|^p
  e^{-\frac{\alpha p}2|z|^{2\ell}}dV(w).\qedhere
\end{align*}

	\end{proof}

\begin{lem}\label{lem:atomic1}
Let $\{z_k\}$ be a sequence satisfying the properties in Lemma \ref{lem:cover}. Then, for $1\le p< \infty$ the map
\[
\{c_k\}\longmapsto \Phi(\{c_k\})(z):=\sum_k c_k\frac{K_{\beta}(z,z_k)}{\|K_{\beta}(z,z_k)\|_{F^{p}_{\beta,\rho}}}
\]
is bounded from the sequence space $\ell^p$ to $F^{p}_{\beta,\rho}$.
\end{lem}

\begin{proof}
	For $p=1$ the result is clear. Assume $p>1$.
	By Corollary \ref{cor:dualF}, the dual of the space 
	$F^{p'}_{\beta,-\rho}$ with respect to the
	 pairing $\langle\cdot,\cdot\rangle_\beta$ 
	 is $F^{p}_{\beta,\rho }$.
	Since the overlapping of the balls $B_k$ is finite, 
Proposition \ref{prop:pnormBergman} and Lemma \ref{lem:propertiescover}\eqref{item:propertiescover2} show that the map
	\[
	g\longmapsto T_{p'}(g):= \bigl\{g(z_k)/
\|K_{\beta}(z,z_k)\|_{F^{p}_{\beta,\rho}}
\bigr\}
	\]
	is bounded from $F^{p'}_{\beta,-\rho}$ to $\ell^{p'}$. Indeed,
\begin{align*}
\|T_{p'}(g)\|_{\ell^{p'}}^{p'}
&\simeq \sum_k |g(z_k)|^{p'}(1+|z_k|)^{-\rho p'-2n(\ell-1)}
e^{-\frac{\beta}2|z_k|^{2\ell}}\\
&\lesssim\sum_k \int_{B_r(z_k)} |g(z)|^{p'}(1+|z|)^{-\rho p'}
e^{-\frac{\beta}2|z|^{2\ell}}dV(z)\simeq \|g\|_{F^{p'}_{\beta,-\rho}}^{p'}
\end{align*}

	So the adjoint map $T^*_{p'}$ of $T_{p'}$, with respect to the pairing
	 $\langle\cdot,\cdot\rangle_\beta$,
	 is bounded from $\ell^p$ to $F^{p}_{\beta,\rho}$. 
	 We are going to show that  $T^*_{p'}=\Phi$.
For $\{c_k\}\in c_{oo}$ (the space of sequences with a finite number of non-zero terms) and $g\in F^{p'}_{\beta,-\rho}$  we have 
	 \begin{align*}
	 	\langle T^*_{p'}\{c_k\},g\rangle_\beta
	 	&=\langle\{c_k\},g(z_k)/
\|K_{\beta}(z,z_k)\|_{F^{p}_{\beta,\rho}}\rangle_{\ell^2}\\
&=\Big\langle \sum_k c_kK_{\beta}(z,z_k)
/\|K_{\beta}(z,z_k)\|_{F^{p}_{\beta,\rho}} ,g\Big\rangle_\beta,
	 	\end{align*}
	since $g(z_k)=\int_{\C^n} g(z)K_{\beta}(z_k,z)e^{-\frac\beta 2|z|^{2\ell}}dV(z)$.
Therefore
	 \[
		T^*_{p'}\{c_k\}=\sum_k c_k\frac{K_{\beta}(z,z_k)}{\|K_{\beta}(z,z_k)\|_{F^{p}_{\beta,\rho}}}\quad (\{c_k\}\in c_{oo}).
\]
Since $c_{oo}$ is dense in $\ell^p$ we conclude that 
$	T^*_{p'}=\Phi$.
	\end{proof}

\begin{proof}[Proof of Proposition \ref{prop:necpgrq}]
Pick $r>0$ satisfying Lemma 
 \ref{lem:propertiescover}
\eqref{item:propertiescover1}, and 
let $\{z_k\}$ be a sequence as  in Lemma \ref{lem:cover}.
Let $\{c_k\}\in \ell^p$ and consider the function 
\[
\Phi_t(\{c_k\})(z)
:=\sum_k c_kr_k(t)\frac{K_{\beta,n}(z,z_k)}
{\|K_{\beta,n}(z,z_k)\|_{F^{p}_{\beta,\rho}}},
\quad 0\le t\le 1,
\]
where $\{r_k(t)\}$ is a sequence of Rademacher functions (see \cite[p.336]{luecking}).
By the hypothesis and Lemma \ref{lem:atomic1}, 
\[
\|\Phi_t(\{c_k\})\|_{F^{q}_{\beta,\eta}}\lesssim
 \|\Phi_t(\{c_k\})\|_{F^{p}_{\beta,\rho}}\lesssim\|\{c_kr_k(t)\}\|_{\ell^p}=\|\{c_k\}\|_{\ell^p}.
 \]
So, by Fubini's theorem and Khinchine's inequality (see \cite[p.336]{luecking})
\begin{align*}
\int_{\C^n}&\left( \sum_{k}|c_k|^2\frac{|K_{\beta,n}(z,z_k)|^2}
{\|K_{\beta,n}(z,z_k)\|^2_{F^{p}_{\beta,\rho}}}
(1+|z|)^{2\eta}e^{-\beta|z|^{2\ell}}\right)^{q/2} dV(z)\\
&\simeq 
\int_0^1\|\Phi_t(\{c_k\})\|_{F^{q}_{\beta,\eta}}^qdt
\lesssim\|\{c_k\}\|_{\ell^p}^{q}.
\end{align*}
By Proposition \ref{prop:pnormBergman} this is equivalent to the fact that 
$I(\{c_k\})\lesssim\|\{c_k\}\|_{\ell^p}^{q}$, where 
\[
I(\{c_k\}):=\int_{\C^n}\left( \sum_{k}|c_k|^2\frac{|K_{\beta,n}(z,z_k)|^2
e^{-\beta|z_k|^{2\ell}}e^{-\beta|z|^{2\ell}}}{
(1+|z_k|)^{2(\rho-\eta) +4n(\ell-1)/p'}}
\right)^{q/2}dV(z).
\]
Now
\[
I(\{c_k\})\gtrsim \int_{\C^n}\left( \sum_{k}|c_k|^2\frac{|K_{\beta,n}(z,z_k)|^2
	e^{-\beta|z_k|^{2\ell}}e^{-\beta|z|^{2\ell}}}{
	(1+|z_k|)^{2(\rho-\eta) +4n(\ell-1)/p'}}\mathcal{X}_{B_k}(z)
\right)^{q/2}dV(z).
\]
Since, by Lemma \ref{lem:cover}, any point $z\in\C^n$ is at most in $N$ balls $B_k$, the equivalence of the $\ell^2$-norm and $\ell^{q/2}$-norm  on $\C^N$ give
\[
I(\{c_k\})\gtrsim \sum_{k} |c_k|^q \int_{B_k} \frac{|K_{\beta,n}(z,z_k)|^q
	e^{-\frac{\beta q}2|z_k|^{2\ell}}e^{-\frac{\beta q}2|z|^{2\ell}}}{
	(1+|z_k|)^{(\rho-\eta) q+2n(\ell-1)q/p'}}
dV(z).
\]
By Lemma \ref{lem:propertiescover}\eqref{item:propertiescover1} 
\[
|K_{\beta,n}(z,z_k)|^q
e^{-\frac{\beta q}2|z_k|^{2\ell}}
e^{-\frac{\beta q}2|z|^{2\ell}}
\simeq (1+|z_k|)^{2n(\ell-1)q}\quad (z\in B_k).
\]
Hence
\begin{align*}
\|\{c_k\}\|_{\ell^p}^{q}
&\gtrsim \sum_{k} |c_k|^q 
	(1+|z_k|)^{-(\rho-\eta) q-2n(\ell-1)(q/p'-q+1)}\\
	&=\sum_{k} |c_k|^q 
	(1+|z_k|)^{-(\rho-\eta) q-2n(\ell-1)(p-q)/p},
\end{align*}
and consequently for any $\{d_k\}\in\ell^{p/q}$, 
\[
\sum_{k} |d_k|
(1+|z_k|)^{-(\rho-\eta)  q-2n(\ell-1)(p-q)/p}\lesssim \|d_k\|_{\ell^{p/q}}.
\]
By the duality of the sequence spaces $(\ell^{p/q})^*=\ell^{p/(p-q)}$, we obtain
\[
\sum_{k} 
(1+|z_k|)^{-(\rho-\eta)  \frac{pq}{p-q}-2n(\ell-1)}<\infty
\]
Since 
\begin{align*}
\infty>
\sum_{k} 
(1+|z_k|)^{-(\rho-\eta)  \frac{pq}{p-q}-2n(\ell-1)}
&\simeq\sum_{k}
\int_{B_k} (1+|z|)^{-(\rho-\eta)  \frac{pq}{p-q}}dV(z)\\
&\simeq \int_{\C^n} (1+|z|)^{-(\rho-\eta)  \frac{pq}{p-q}}
dV(z),
\end{align*}
we conclude that $-(\rho-\eta)  \frac{pq}{p-q}<-2n$.
This ends the proof.
\end{proof}

\subsection{Necessary condition for $1\le q<p=\infty$ and $\beta=\gamma$}\quad\par

In this section we extend Proposition 
 \ref{prop:necpgrq}
to the case $p=\infty$.

\begin{prop}\label{prop:necinf>q}
If $1\le q<\infty$ and  
$F^{\infty}_{\beta,\rho}
\hookrightarrow F^{q}_{\beta,\eta}$ 
then
$
\frac {2n}q< \rho-\eta.
$
\end{prop}

The necessary condition will be obtained from the
 case $1\le q<p<\infty$ by complex interpolation. 
In particular we will use the Riesz-Thorin theorem 
and the  following well-known result (see for instance 
\cite[Lemma 7.11]{kalton-mayboroda-mitrea}).

\begin{lem}\label{lem:retract} 
Let $(Y_0,Y_1)$  and $(X_0,X_1)$ be admissible
 pairs of Banach spaces. 
Assume that $(Y_0,Y_1)$ is a
 retract of $(X_0,X_1)$, that is, there exist
 bounded linear operators $E:Y_j\to X_j$ 
and $R:X_j\to Y_j$ such that 
$R\circ E$ is the identity operator on $Y_j$,  $j=0,1$. 
 Then
$(Y_0,Y_1)_{[\theta]}=R((X_0,X_1)_{[\theta]})$.
\end{lem}

\begin{lem}\label{lem:interpolation}
Let $1\le q<\infty$ and let $\theta\in(0,1)$. 
If $\frac 1s=\frac{1-\theta}q$ then
\[
(F^{q}_{\beta,\rho},
 F^{\infty}_{\beta,\rho})_{[\theta]}
=F^{s}_{\beta,\rho}
\quad\text{and}\quad 
(F^{q}_{\beta,\rho},
 F^{q}_{\beta,\eta})_{[\theta]}
=F^{q}_{\beta,(1-\theta)\rho+\theta\eta}.
\]
\end{lem}

\begin{proof}
Observe that the map 
$\Phi(f)(z)
:= f(z)e^{\frac\beta 2|z|^{2\ell}}(1+|z|)^{-\rho}$
 is a linear isometry 
from $L^r$ onto $L^{r}_{\beta,\rho}$, $1\le r\le\infty$.
So by Lemma \ref{lem:interpolation}
 and the  Riesz-Thorin theorem, we obtain 
\[
(L^{q}_{\beta,\rho}, L^{\infty}_{\beta,\rho})_{[\theta]}
=\Phi((L^q,L^\infty)_{[\theta]})
=\Phi(L^s)=L^{s}_{\beta,\rho}
\]
By Proposition \ref{prop:PLpontoFp}, for $1\le r\le\infty$,  $(F^{q}_{\beta,\rho}, F^{\infty}_{\beta,\rho})$ is a retract of $(L^{q}_{\beta,\rho}, L^{\infty}_{\beta,\rho})$ and so 
\[
(F^{q}_{\beta,\rho}, F^{\infty}_{\beta,\rho})_{[\theta]}
=P_\beta((L^{q}_{\beta,\rho}, L^{\infty}_{\beta,\rho})_{[\theta]})
=P_\beta(L^{s}_{\beta,\rho})=F^{s}_{\beta,\rho},
\]
which proves the first interpolation identity.

In order to prove the second identity, by Theorem \cite[Theorem 5.5.3]{berg-lofstrom} we have
\begin{align*}
(L^{q}_{\beta,\rho}, L^{q}_{\beta,\eta})_{[\theta]}
&=(L^q(e^{-\frac{q\beta} 2|z|^{2\ell}}(1+|z|)^{q\rho}),
 L^q(e^{-\frac{q\beta}  2|z|^{2\ell}}(1+|z|)^{q\eta}))_{[\theta]}\\
&=L^q(e^{-\frac{q\beta}  2|z|^{2\ell}}(1+|z|)^{q((1-\theta)\rho+\theta\eta)})=
L^{q}_{\beta,(1-\theta)\rho+\theta\eta}.
\end{align*}
Therefore, as above, 
\[
(F^{q}_{\beta,\rho}, F^{q}_{\beta,\eta})_{[\theta]}
=P_\beta(L^{q}_{\beta,(1-\theta)\rho+\theta\eta})
=F^q(e^{-\frac\beta 2|z|^{2\ell}}(1+|z|)^{1-\theta)\rho+\theta\eta}).
\]
This ends the proof.
\end{proof}

\begin{proof}[Proof of Proposition \ref{prop:necinf>q}]

Assume $ F^{\infty}_{\beta,\rho}\hookrightarrow F^{q}_{\beta,\eta}$. By Lemma \ref{lem:interpolation},
\[
F^{s}_{\beta,\rho}=(F^{q}_{\beta,\rho}, F^{\infty}_{\beta,\rho})_{[\theta]}
\hookrightarrow
(F^{q}_{\beta,\rho}, F^{q}_{\beta,\eta})_{[\theta]}
=F^{q}_{\beta,(1-\theta)\rho+\theta\eta},
\]
with $\frac 1s=\frac{1-\theta}q$.
Since $q=(1-\theta)s<s<\infty$, Proposition \ref{prop:necpgrq}  gives 
\[
2n(\tfrac  1q-\tfrac 1s)<\rho-((1-\theta)\rho+\theta\eta)=
q(\tfrac 1q-\tfrac 1s)(\rho-\eta),
\]
and so $\frac{2n}q<\rho-\eta$.
\end{proof}

\subsection{Proof of Theorem \ref{thm:embeddings} for $1\le q<p\le\infty$} \quad\par

The sufficient conditions follow from Lemma \ref{lem:condsuffip>q}.

If $\beta\ne \gamma$ the necessary condition 
 $\beta<\gamma$ follows from  Lemma 
 \ref{lem:necessarypq}.
If $\beta= \gamma$ the necessary condition follows 
 from Propositions \ref{prop:necpgrq} and 
 \ref{prop:necinf>q}.

\section{Proof of Theorem \ref{thm:Bprojection}}
\label{sec:boundedness}

First we prove the necessary condition 
 $\beta<2\alpha$. For the case $\rho=0$ next lemma corresponds to \cite[Lemma 3]{bommier-englis-youssfi}.

\begin{lem}\label{lem:necbeta<2alpha}
Let $1\le p,q\le\infty$. 
If $P_\alpha$ is bounded from 
 $(L^{2}_\alpha\cap L^{p}_{\beta,\rho},
\|\cdot\|_{L^{p}_{\beta,\rho}})$ 
 to $L^{q}_{\gamma,\eta}$ then $\beta<2\alpha$.
	\end{lem}

\begin{proof}
	For any $\z\in\C^n$, the linear form $g\mapsto g(\z)$ 
is bounded on $F^{q}_{\gamma,\eta}$ 
(see Corollary \ref{cor:embedFinfty}). Then 
the boundedness of 
$P_\alpha:(L^{2}_\alpha\cap L^{p}_{\beta,\rho},
\|\cdot\|_{L^{p}_{\beta,\rho}})\mapsto
 L^q_{\gamma,\eta}$
 implies the boundedness of the form 
 $U_\z(f)=P_\alpha(f)(\z)$ on 
$(L^{2}_\alpha\cap L^{p}_{\beta,\rho},
\|\cdot\|_{L^{p}_{\beta,\rho}})$.
 Hence  Lemma \ref{lem:well-defined} gives
 $\beta<2\alpha$.
\end{proof}

Now the proof of Theorem \ref{thm:Bprojection}
 follows from the next proposition and its corollary.

\begin{prop}\label{prop:Ponto}
	Let  $1\le p\le \infty$.
	If $0<\beta<2\alpha$ then the Bergman projection
	 $P_\alpha$ is bounded from  $L^{p}_{\beta,\rho}$ 
	 onto $F^{p}_{\alpha^2/(2\alpha-\beta),\rho}$.
\end{prop}

\begin{cor}\label{cor:bounded-embed}
Let $1\le p,q\le\infty$ and let $0<\beta<2\alpha$.  
Then the Bergman projection $P_\alpha$ 
is bounded from $L^{p}_{\beta,\rho}$ to
 $L^{q}_{\gamma,\eta}$ if 
and only if
$F^{p}_{\alpha^2/(2\alpha-\beta),\rho}
	   \hookrightarrow F^{q}_{\gamma,\eta}$.
\end{cor}

Taking for granted these results, we finish the proof of Theorem \ref{thm:Bprojection}.

\begin{proof}[Proof of Theorem \ref{thm:Bprojection}]
By Lemma \ref{lem:necbeta<2alpha} it is clear that 
$\beta<2\alpha$ is a necessary condition for the
 boundedness of  $P_\alpha$ from
 $L^{p}_{\beta,\rho}$ to $L^{q}_{\gamma,\eta}$.

If $\beta<2\alpha$, Corollary \ref{cor:bounded-embed}
 shows that 
$P_\alpha$ is bounded from
 $L^{p}_{\beta,\rho}$ to $L^{q}_{\gamma,\eta}$ 
if and only $F^{p}_{\alpha^2/(2\alpha-\beta),\rho}
	   \hookrightarrow F^{q}_{\gamma,\eta}$.
Thus Theorem \ref{thm:Bprojection} is a consequence
 of Theorem \ref{thm:embeddings}.
\end{proof}

We conclude this section with the proofs of 
Proposition \ref{prop:Ponto}
and Corollary \ref{cor:bounded-embed}. 
To do so, we introduce the following notations which
 will used in the next results.
For $\beta<2\alpha$, let 
\[
\delta:=\bigl(\tfrac{\alpha}{2\alpha-\beta}\bigr)^{1/\ell}
\quad\text{and}\quad 
\kappa:=\alpha\delta^\ell
=\tfrac{\alpha^2}{2\alpha-\beta}.
\]

The next lemma follows from \eqref{eqn:Kdelta}.

\begin{lem}\label{lem:PatoPk}
If $f\in L^{p}_{\beta,\rho}$, then 
$P_\alpha(f)=P_\kappa(T_\delta(f))$, where
\[
T_\delta(f)(z)
=\delta^n f(\delta z)e^{(\alpha-\beta)|\delta z|^{2\ell}}.
\]
\end{lem}

\begin{proof}
Using the change of variables $w=\delta u$ and 
 \eqref{eqn:Kdelta}, we obtain 
\begin{align*}
P_\alpha(f)(z)
&=\delta^{2n}\int_{\C^n} f(\delta u) K_\alpha(z,\delta u)
 e^{-\alpha |\delta u|^{2\ell}}dV(u)\\
&=\delta^{n}\int_{\C^n} [f(\delta u) 
e^{(-\alpha +\kappa\delta^{-2\ell})|\delta u|^{2\ell}}]
 K_\kappa(z,u) e^{-\kappa | u|^{2\ell}}dV(u).
\end{align*}
Since $-\alpha+\kappa\delta^{-2\ell}
=-\alpha+\alpha\delta^{-\ell}
=-\alpha+2\alpha-\beta=\alpha-\beta$ 
we obtain the result.
\end{proof}

\begin{lem}\label{lem:operatorT}
The operator $T_\delta$ is a topological  isomorphism
 from $L^{p}_{\beta,\rho}$ onto $L^{p}_{\kappa,\rho}$.
\end{lem}

\begin{proof}
Since 
 $\alpha-\beta=-\frac{\beta}2+\frac{2\alpha-\beta}2
=-\frac{\beta}2+\frac \kappa 2 \delta^{-2\ell}$, 
we have
\[
T_\delta(f)(z)=\delta^n f(\delta z)
e^{-\frac{\beta}2|\delta z|^{2\ell}} 
e^{\frac{\kappa}2| z|^{2\ell}}.
\]
Therefore
\begin{align*}
\|T_\delta(f)\|_{L^{p}_{\kappa,\rho}}
&\simeq \|f(\delta z)
e^{-\frac{\beta}2|\delta z|^{2\ell}}(1+|z|)^\rho\|_{L^p}\\
&\simeq \|f(\delta z)
e^{-\frac{\beta}2|\delta z|^{2\ell}}
(1+|\delta z|)^\rho\|_{L^p}
\simeq \|f\|_{L^{p}_{\beta,\rho}}.
\end{align*}

So to conclude the proof we only need to show that 
the operator $T_\delta$ is surjective. 
This follows from the fact that the unique solution 
of the equation $T_\delta (f)=g$ is 
 $f(z)=\delta^{-n}g(z/\delta)e^{(\beta-\alpha)|z|^{2\ell}}$
and  
\[\|f\|_{L^{p}_{\beta,\rho}}
\simeq \|T_\delta(f)\|_{L^{p}_{\kappa,\rho}}
=\|g\|_{L^{p}_{\kappa,\rho}}.\qedhere
\]
\end{proof}

\begin{proof}[Proof of Proposition \ref{prop:Ponto}]

By Proposition \ref{prop:PLpontoFp}, $P_\kappa$ 
is a bounded operator from $L^{p}_{\kappa,\rho}$
 onto $F^{p}_{\kappa,\rho}$.
So Lemmas \ref{lem:PatoPk} and \ref{lem:operatorT}
give
\[
P_{\alpha}(L^{p}_{\beta,\rho})
=P_{\kappa}(T_\delta(L^{p}_{\beta,\rho}))
=P_{\kappa}(L^{p}_{\kappa,\rho})
=F^{p}_{\kappa,\rho}.\qedhere
\]

\end{proof}

\begin{proof}[Proof of Corollary 
 \ref{cor:bounded-embed}]
By Proposition \ref{prop:Ponto}, it is clear  that if
$F^{p}_{\alpha^2/(2\alpha-\beta),\rho}
\hookrightarrow F^{q}_{\gamma,\eta}$, 
then $P_\alpha$
is bounded from $L^{p}_{\beta,\rho}$ to 
 $L^{q}_{\gamma,\eta}$.

Conversely, if  $P_\alpha$ is bounded from 
 $L^{p}_{\beta,\rho}$ to $L^{q}_{\gamma,\eta}$ then,  
by Proposition \ref{prop:Ponto},
\[
F^{p}_{\alpha^2/(2\alpha-\beta),\rho}
=P_\alpha(L^{p}_{\beta,\rho})
\hookrightarrow F^{q}_{\gamma,\eta}.\qedhere
\]
	\end{proof}

\end{document}